\documentclass{article}
\usepackage[letterpaper, left=1in, right=1.5in, vmargin=1in]{geometry}\usepackage[utf8]{inputenc}
\usepackage{amsmath,amssymb,amsthm}
\usepackage{enumitem}
\usepackage[noabbrev]{cleveref}

\usepackage{titlesec}
\titlelabel{\thetitle.\quad}

\usepackage[textsize=scriptsize]{todonotes}
\newtheorem{theorem}{Theorem}
\newtheorem{corollary}[theorem]{Corollary}

\newtheorem{proposition}[theorem]{Proposition}

\title{\textbf{Polynomial Identities for Generalized Fibonacci Sequences}}
\author{Joey Glampe, Aryaman Joshi, Swagat Malla, Matthew Richey, Matthew Wright
\\[1em] St.\ Olaf College (Northfield, MN)}

\begin{document}

\maketitle

\begin{abstract}
    The Fibonacci, Lucas, and Pell sequences satisfy certain families of ``polynomial'' identities. 
    While a few members of these families are well-known, we are not able to find a general statement of these identities in the mathematical literature.
    We show that such polynomial identities exist for many generalized Fibonacci sequences, and we classify sequences that have these identities obtained by two distinct approaches.
    One approach uses the Girard-Waring formula, while the other relies on Binet's formula and linear algebra.
\end{abstract}

\section{Background}
The Fibonacci sequence is defined by $F_0 = 0$, $F_1 = 1$, and the recurrence $F_{n} = F_{n-1} + F_{n-2}$ for integers $n \ge 2$.
The Fibonacci numbers satisfy a family of ``polynomial'' identities, which allow $F_{kn}$ to be expressed as a polynomial in $F_n$ for any positive odd integer $k$.
For example, when $k=3$ we have the well-known identity
\begin{equation}\label{fib3n}
  F_{3n} = 5F_n^3 + (-1)^n 3 F_n.
\end{equation}
When $k=5$ we obtain
\begin{equation}\label{fib5n}
    F_{5n} = 25F_n^5 + (-1)^n 25 F_n^3 + 5F_n.
\end{equation}
Observe that in each of \cref{fib3n,fib5n}, the right-hand side is a polynomial in $F_n$, though there is one coefficient that changes sign according to the parity of $n$ due to the factor of $(-1)^n$.
\Cref{fib3n} appears in Koshy's text \cite[page 110]{koshy2017} with reference to Halton \cite{halton1965}.
Jennings states \cref{fib3n,fib5n}, along with a similar identity for $F_{7n}$, and then shows that for every odd $k$, $F_{kn}$ can be expressed as a polynomial in $F_n$ \cite{jennings1993}. His proof relies on the Binet formula for the Fibonacci numbers and does not hold when $k$ is even.

Similar identities exist for the Lucas numbers, but for these $k$ may be any positive integer.
The Lucas sequence is defined $L_0 = 2$, $L_1 = 1$, and $L_{n} = L_{n-1} + L_{n-2}$ for $n \ge 2$.
Squaring the Lucas numbers leads to the simple identity
\begin{equation}
    L_{2n} = L_n^2 - 2(-1)^n.
\end{equation}
Higher powers of Lucas numbers satisfy identities such as
\begin{align}
    L_{3n} &= L_n^3 - 3(-1)^n L_n, \\
    L_{4n} &= L_n^4 - 4(-1)^n L_n^2 + 2,
\end{align}
which can be obtained from the identities in exercises 61 and 80 on page 118 of Koshy's text \cite{koshy2017}.
We call these \emph{polynomial identities} since the right-hand side is a polynomial in $L_n$, though again the sign of certain coefficients depends on the parity of $n$.
Similar identities with the left side $L_{kn}$ for all positive integers $k$ are possible.

This immediately leads to the question: Why do the Lucas numbers admit polynomial identities for all positive integers $k$, while the Fibonacci numbers only admit such identities for odd $k$?
We will return to this question after further generalization.

The Pell numbers \emph{also} satisfy polynomial identities.
The Pell sequence is defined $P_0 = 0$, $P_1 = 1$, and $P_{n} = 2P_{n-1} + P_{n-2}$ for $n \ge 2$.
Martinjak and Vrsaljko state the identities
\begin{align}
    P_{3n} &= 8P_n^3 + (-1)^n 3 P_n, \\
    P_{5n} &= 64P_n^5 + (-1)^n 40 P_n^3 + 5P_n,
\end{align}
which they generalize to a closed-form expression for $P_{kn}$ in terms of $P_n$ when $k$ is odd \cite{martinjak2017}.
Furthermore, they consider recursive sequences $(A_n)$ defined by $A_0=0$, $A_1=1$, and $A_{n} = sA_{n-1} + A_{n-2}$, generalizing both the Fibonacci and Pell sequences. 
They show that any such sequence satisfies the identity
\begin{equation}\label{mvPolyId}
    A_{(2m+1)n} = \sum_{i=0}^m (-1)^{n(m+i)} (s^2 + 4)^i \frac{2m + 1}{2i + 1} \binom{m+i}{2i} A_n^{2i+1}
\end{equation}
for positive integers $m$.
\Cref{mvPolyId} expresses $A_{(2m+1)n}$ as a polynomial in $A_n$, where the coefficients of the polynomial are allowed to include a factor of $(-1)^n$.

Having encountered identities of a particular polynomial form for Fibonacci, Lucas, and Pell sequences, as well as the more general $(A_n)$ sequence above, we are led to ask whether such identities exist for other generalized Fibonacci sequences.
Let $a$, $b$, $r$, and $s$ be integers with $a$ and $b$ not both zero, and $r$ and $s$ both nonzero.
Define a generalized Fibonacci sequence, or \emph{gibonacci sequence}, $(G_n)$ by
\begin{equation}\label{gibSeqDef}
    G_0 = a,\quad G_1 = b,\quad \text{and}\quad G_n = rG_{n-1} + sG_{n-2} \text{ for } n \ge 2.
\end{equation} 
Our questions are: Do all such gibonacci sequences have identities of the form
\begin{equation}\label{gibPoly}
    G_{kn} = a_k G_n^k + a_{k-1}G_n^{k-1} + \cdots + a_2 G_n^2 + a_1 G_n + a_0,
\end{equation}
where the coefficients $a_i$ may include a $(-1)^n$ factor?
If this is not possible for all gibonacci sequences, for which sequences \emph{is} it possible?
Given a gibonacci sequence, for which values of $k$ do such identities exist?
How are the coefficients obtained? 
Is every other coefficient always zero, as is the case in the identities presented above?

We find that polynomial identities exist for some, but not all, gibonacci sequences of the form defined above.
The Fibonacci, Lucas, and Pell polynomial identities stated above can be derived from the Girard-Waring formula, and these identities generalize to certain other gibonacci sequences.
Another set of polynomial identities for other gibonacci sequences can be obtained by a careful application of Binet's formula and linear algebra.
Further polynomial identities exist for the special cases of geometric and arithmetic sequences.

\section{The Girard-Waring formula}

Our first derivation of gibonacci polynomial identities begins with the Girard-Waring formula.
For any real numbers $x$ and $y$, and positive integer $k$, 
\begin{equation}\label{waring}
    x^k + y^k = \sum_{i=0}^{\left\lfloor k/2 \right\rfloor} (-1)^i \frac{k}{k-i} \binom{k-i}{i} (x+y)^{k-2i} (xy)^i.
\end{equation}
This identity was found by Girard in 1629 and often attributed to Waring, who stated it in 1762 \cite{gould1997}.

The first use of the Girard-Waring formula to derive polynomial identities involving the Fibonacci numbers may have been by Filipponi, who also used this formula to obtain polynomial relationships between Lucas and Fibonacci numbers \cite{filipponi1995}.
Swamy extended Filipponi's work, using another combinatorial identity to establish new relations between Lucas and Fibonacci numbers \cite{swamy1997}. 
However, the Girard-Waring formula is more broadly applicable, resulting in polynomial identities for many other (but not all) gibonacci sequences.

Consideration of exponents provides some intuition for why the Girard-Waring formula may be promising for our purposes. 
Recall that each of the polynomial identities stated above for the Fibonacci, Lucas, and Pell sequences involves only exponents of the same parity. On the right side of \cref{waring}, the factor $(x+y)^{k-2i}$ will also have exponents of the same parity since $i$ is an integer. Thus, if $x+y = G_n$, \cref{waring} might yield a polynomial identity---an idea which is key to the proof of the following theorem.

\begin{theorem}\label{gibPolyThm}
    Define sequence $(G_n)$ by $G_0 = a$, $G_1 = b$, and $G_n = rG_{n-1} + sG_{n-2}$ for $n \ge 2$, where $a$ and $b$  are not both zero, $r$ and $s$ are both nonzero, and $s \ne \frac{-r^2}{4}$.
    Then $(G_n)$ satisfies an identity of the form
    \begin{equation}\label{polyIdForm}
        G_{kn} = \sum_{i=0}^{\left\lfloor k/2 \right\rfloor}  (-s)^{ni} a_i G_n^{k-2i}
    \end{equation}
    for some coefficients $a_i$ (that depend only on the sequence $(G_n)$ and $k$) if one of the following cases holds:
    \begin{enumerate}[label=\bf{Case \arabic*.},leftmargin=46pt]
        \item $(G_n)$ is a geometric sequence and $k$ is a positive integer;
        \item the initial term $a=0$ and $k$ is a positive odd integer;
        \item $a$ divides $2b$, $r = \frac{2b}{a}$, and $k$ is a positive integer.
    \end{enumerate}
\end{theorem}

We emphasize that the coefficients $a_i$ in \Cref{gibPolyThm} do not depend on $n$.

\begin{proof}[Proof of \Cref{gibPolyThm}]
    The recurrence relation that defines $(G_n)$ has characteristic equation $t^2 - rt - s = 0$.
    The roots of the characteristic polynomial are 
    \[ t_{1,2} = \frac{r \pm \sqrt{r^2 + 4s}}{2} .\]
    Since $s \ne \frac{-r^2}{4}$, these roots are distinct.
    Thus, the terms of the sequence can be written using the Binet formula
    \[ G_n = c_1 t_1^n + c_2 t_2^n \] 
    for some coefficients $c_1$ and $c_2$.
    Note that since $s$ is nonzero, the roots $t_1$ and $t_2$ are also both nonzero.
    Using the initial values $a = c_1 + c_2$ and $b = c_1t_1 + c_2 t_2$, we find
    \[ c_1 = \frac{2b -ar + a\sqrt{r^2 + 4s}}{2\sqrt{r^2 + 4s}}  = \frac{b-at_2}{\sqrt{r^2 + 4s}}\]
    and
    \[ c_{2} = \frac{ar+ a\sqrt{r^2 + 4s} -2b}{2\sqrt{r^{2} + 4s}} = \frac{at_{1} -b}{\sqrt{r^2 + 4s}}.\]
    Assume without loss of generality that $c_1 \ne 0$ (otherwise swap indexes).

    Let $x = c_1t_1^n$ and $y = c_2t_2^n$.
    Substituting these values of $x$ and $y$ into the Girard-Waring formula (\cref{waring}), we obtain
    \begin{equation}\label{waring2}
        (c_1t_1^n)^k + (c_2t_2^n)^k = \sum_{i=0}^{\left\lfloor k/2 \right\rfloor} (-1)^i \frac{k}{k-i} \binom{k-i}{i} (c_1t_1^n + c_2t_2^n)^{k-2i} (c_1t_1^n c_2t_2^n)^i.
    \end{equation}
    On the left side, we can isolate the quantity $c_1 t_1^{kn} + c_2 t_2^{kn}$, which equals $G_{kn}$.
    To simplify the right side, note that $c_1 t_1^n + c_2 t_2^n = G_n$ and $t_1t_2 = -s$.
    The identity then becomes
    \begin{equation}\label{gibWaring}
        c_1^{k-1} G_{kn} + c_2\left( c_2^{k-1} - c_1^{k-1}\right) t_2^{kn} = \sum_{i=0}^{\left\lfloor k/2 \right\rfloor} (-1)^i \frac{k}{k-i} \binom{k-i}{i} (c_1c_2)^i (-s)^{ni} G_n^{k-2i}.
    \end{equation}

    Observe that \cref{gibWaring} is nearly in the desired form, but the left side contains the term $c_2\left( c_2^{k-1} - c_1^{k-1}\right) t_2^{kn}$, which we would like to eliminate.
    This term vanishes in three cases.
    \begin{enumerate}[label=\textbf{Case \arabic*:},leftmargin=55pt]
        \item If $c_2=0$, then the sequence $(G_n)$ is in fact a geometric sequence given by $G_n = c_1t_1^n$.
            In this case, \cref{gibWaring} reduces to the identity
            \begin{equation}\label{idCase1}
                G_{kn} = \frac{1}{c_1^{k-1}} G_n^k
            \end{equation}
            for all positive integers $k$.
            This is an identity in the form of \cref{polyIdForm} with $a_0 = \frac{1}{c_1^{k-1}}$ and all other $a_i = 0$.

    \item If $c_2 = -c_1$, then $\left( c_2^{k-1} - c_1^{k-1}\right) = 0$ when $k$ is odd.
        In this case, let $c = c_1 = -c_2$, and note that this implies the initial term $G_0 = a = 0$.
        Then \cref{gibWaring} reduces to the identity
        \begin{equation}\label{idCase2}
            G_{kn} = \sum_{i=0}^{\left\lfloor k/2 \right\rfloor} \frac{k}{k-i} \binom{k-i}{i} c^{2i-k+1} (-s)^{ni} G_n^{k-2i}
        \end{equation}
        for all positive odd integers $k$.
        This is an identity in the form of \cref{polyIdForm} with $a_i = \frac{k}{k-i} \binom{k-i}{i} c^{2i-k+1}$, for any sequence $(G_n)$ as defined above whose initial term is zero.
        
    \item If $c_2 = c_1$, then $\left( c_2^{k-1} - c_1^{k-1}\right) = 0$ and for all $k$.
        In this case, let $c = c_1 = c_2$. 
        Then \cref{gibWaring} reduces to the identity
        \begin{equation}\label{idCase3}
            G_{kn} = \sum_{i=0}^{\left\lfloor k/2 \right\rfloor} (-1)^i \frac{k}{k-i} \binom{k-i}{i} c^{2i-k+1} (-s)^{ni} G_n^{k-2i}
        \end{equation}
        for all positive integers $k$.
        This is an identity in the form of \cref{polyIdForm} with $a_i = (-1)^i \frac{k}{k-i} \binom{k-i}{i} c^{2i-k+1}$.

        Now $c_1 = c_2 = c$ exactly when the initial values are $a = 2c$ and $b = c(t_1 + t_2) = cr$, which implies that $r = \frac{2b}{a}$.
        Since we require $r$ to be an integer, we must have that $a$ divides $2b$.
        Thus, whenever $a$ divides $2b$ and $r = \frac{2b}{a}$, \cref{idCase3} holds for all positive integers $k$. 
        \qedhere
    \end{enumerate}
\end{proof}

Note that the identity in \Cref{gibPolyThm} is more general than what we set out to find, due to the factor $(-s)^{ni}$.
Thus, we may regard the identities in \cref{idCase2,idCase3} as \emph{generalized} polynomial identities, since their coefficients depend on $s$. 
However, if $s = \pm 1$, we do in fact obtain polynomial identities in the form of \cref{{gibPoly}}, where the coefficients $a_i$ may include a $(-1)^n$ factor.

In particular, if $s=1$, then $G_{kn}$ is not equal to the same polynomial in $G_n$ for all $n$, but two related polynomials, whose only difference is the signs of every other coefficient.
The Fibonacci, Lucas, and Pell sequences all fall into this pattern.
Moreover, the Fibonacci and Pell sequences exhibit polynomial identities for only odd $k$ because these sequences start with their zero-indexed term equal to zero (Case 2 above).
The Lucas sequence has such identities for all $k$ because it falls into Case 3.

Furthermore, if $s=-1$, the identities in \cref{idCase2,idCase3} are especially nice, since the $(-s)^{nk}$ factor always equals $1$.
In this case, $G_{kn}$ is given by the \emph{same} polynomial in $G_n$ for all $n$.
As an example, let $S_0 = 0$, $S_1 = 1$, and $S_n = 4S_{n-1} - S_{n-2}$. 
This sequence falls into Case 2 of \Cref{gibPolyThm}, so it satisfies the family of identities
\begin{align*}
   S_{3n} &= 12 S_n^3 + 3 S_n \\
   S_{5n} &= 144 S_n^5 + 60 S_n^3 + 5 S_n \\
   S_{7n} &= 1728 S_n^7 + 1008 S_n^5 + 168 S_n^3 + 7 S_n \\
   &\ \,\vdots
\end{align*}
for all nonnegative integers $n$.

\section{A general approach based on Binet's formula}

While the cases identified above generalize our original identities involving the Fibonacci, Lucas, and Pell sequences, we find that other polynomial identities exist for other gibonacci sequences.
For example, consider the sequence defined $S_0=1$, $S_1=1$, and $S_n = 5S_{n-1} + 6S_{n-2}$ for $n \ge 2$.
(This sequence is $1, 1, 11, 61, 371, 2221, 13331, 79981, \ldots$.)
This sequence satisfies the polynomial identity
\begin{equation}\label{gen3n}
  S_{3n} = \frac{49}{4}S_n^3 + (-1)^{n+1}\frac{105}{4}S_n^2 + \frac{75}{4}S_n + (-1)^{n+1}\frac{15}{4}.
\end{equation}
Intriguingly, the coefficients of \cref{gen3n} are not integers, in contrast with the identities examined earlier.
Furthermore, the exponents of $S^n$ in \cref{gen3n} are consecutive integers, rather than integers with the same parity as we have seen previously. 
Evidently this new identity does not arise from the argument involving Waring's formula presented in \Cref{gibPolyThm}; rather, it is a consequence of the next theorem.

\begin{theorem}\label{genpoly}
    Let sequence $(G_n)$ be defined $G_0 = a$, $G_1 = b$, and $G_n = rG_{n-1} + sG_{n-2}$ for $n \ge 2$, where $a$ and $b$  are not both zero, $r$ and $s$ are both nonzero, and $s \ne \frac{-r^2}{4}$.
    Then 
    \begin{equation}\label{polyId1}
        G_{kn} = \sum_{i=0}^{k}  a_i(n) G_n^{i}
    \end{equation}
    where $a_0(n) = c_2 t_2^{kn} + \frac{(-c_2 t_2^{n})^k }{c_1^{k-1}}$ and $a_i(n) = \binom{k}{i} \frac{(-c_2 t_2^{n})^{k-i} }{c_1^{k-1}}$ for $i \in \{1, 2, \ldots, k\}$.
\end{theorem}
\begin{proof}
    The Binet formula for this sequence is $G_n = c_1 t_1^n + c_2 t_2^n$.
    Assume without loss of generality that $c_1 \ne 0$ (otherwise swap indexes).
    Since $s \ne \frac{-r^2}{4}$, roots $t_1$ and $t_2$ are distinct, so Binet's formula allows us to rewrite \cref{polyId1} as
    \begin{equation}\label{proofEq1}
        c_1 t_1^{kn} + c_2 t_2^{kn} = \sum_{i=0}^k a_i(n) \left( c_1 t_1^n + c_2 t_2^n \right)^i.
    \end{equation}
    Expanding the right-hand side, we can rewrite \cref{proofEq1} as
    \begin{equation}\label{proofEq2}
        c_1 t_1^{kn} + c_2 t_2^{kn} = \sum_{i=0}^k \sum_{j=0}^k a_i(n) \binom{i}{j} (c_1 t_1^n)^j (c_2 t_2^{n})^{i-j}.
    \end{equation}
    
    By separating out distinct powers of $t_1$, we can rewrite \cref{proofEq2} as $k+1$ distinct equations, each of which is linear in the coefficients $a_i(n)$.
    Equation \cref{proofEq2} can be then written in matrix form as follows:    
    \begin{equation}\label{proofMxEq}
        \begin{bmatrix}
            c_2 t_2^{kn} \\ 0 \\ 0 \\ \vdots \\ 0 \\ c_1
        \end{bmatrix}
        = 
        \begin{bmatrix}
            1 & c_2 t_2^{n} & c_2^2 t_2^{2n} & \cdots & c_2^{k-1} t_2^{(k-1)n} & c_2^k t_2^{kn} \\ 
            0 & c_1 & 2c_1c_2 t_2^{n} & \cdots & (k-1) c_1 c_2^{k-2} t_2^{(k-2)n} & k c_1 c_2^{k-1} t_2^{(k-1)n} \\
            0 & 0 & c_1^2 &  \cdots & \binom{k-1}{2} c_1^2 c_2^{k-3} t_2^{(k-3)n} & \binom{k}{2} c_1^2 c_2^{k-2} t_2^{(k-2)n} \\
            \vdots & \vdots & \vdots & \ddots & \vdots & \vdots \\
            0 & 0 & 0 & \cdots & c_1^{k-1} & k c_1^{(k-1)} c_2 t_2^{n} \\
            0 & 0 & 0 & \cdots & 0 & c_1^{k} \\
        \end{bmatrix}
        \begin{bmatrix}
            a_0(n) \\ a_1(n) \\ a_2(n) \\ \vdots \\ a_{k-1}(n) \\ a_{k}(n)
        \end{bmatrix}.
    \end{equation}
    In the square matrix above, the entry in row $j$ column $i$ (indexing rows and columns from $0$ to $k$) is 
    \[ \binom{i}{j} c_1^j (c_2 t_2^{n})^{i-j}. \]
    Importantly, since since $c_1 \ne 0$, this matrix has nonzero determinant.
    Therefore, there exist coefficients $a_i(n)$ that satisfy \cref{polyId1}.
    These coefficients are:
    \begin{align*}
       a_0(n) &= c_2 t_2^{kn} + \frac{(-c_2 t_2^{n})^k }{c_1^{k-1}},   \\
       a_i(n) &= \binom{k}{i} \frac{(-c_2 t_2^{n})^{k-i} }{c_1^{k-1}} \quad \text{for} \quad i \in \{1, 2, \ldots, k\}.
    \end{align*}

    To verify this solution, consider row $j \in \{1, 2, \ldots, k-1\}$ of the matrix product in \cref{proofMxEq}. This row of the product is:
    \begin{align*}
        \sum_{i=j}^k \binom{i}{j} c_1^j (c_2 t_2^{n})^{i-j} \cdot a_i(n) 
        &= \sum_{i=j}^k \binom{i}{j} c_1^j (c_2 t_2^{n})^{i-j} \cdot \binom{k}{i} \frac{(-c_2 t_2^{n})^{k-i} }{c_1^{k-1}} \\
        &= \frac{(c_2t_2^n)^{k-j}}{c_1^{k-j-1}} \sum_{i=j}^k (-1)^{k-i} \binom{i}{j}\binom{k}{i} = 0,
    \end{align*}
    where the last equality is due to the binomial coefficient orthogonality identity.
    The justification for row $0$ of the product is essentially the same, except that $a_0(n)$ adds the $c_2t_2^{kn}$ term that is found on the left side of \cref{proofMxEq}.
    For row $k$ of \cref{proofMxEq}, it is straightforward to verify that $c_1^k \cdot a_k(n) = c_1$.
\end{proof}

Again, the identities we obtain from \Cref{genpoly} are more general than what we sought, since the coefficients of the sum on the right side of \cref{polyId1} depend on $n$, in the form of $t_2^n$.
However, if $t_2 = \pm 1$, we obtain identities in the form of \cref{gibPoly}, in which the coefficients may include a $(-1)^n$ factor but no other dependence on $n$.
We state these cases in the following corollaries.

\begin{corollary}
    Let $s = 1-r$, so that the gibonacci sequence $(G_n)$ is defined $G_0 = a$, $G_1 = b$, and $G_n = rG_{n-1} + (1-r)G_{n-2}$ for $n \ge 2$.
    This sequence satisfies the polynomial identity
    \begin{equation}\label{polyId2}
        G_{kn} = c_2 + \sum_{i=0}^{k} \binom{k}{i} \frac{(-c_2)^{k-i} }{c_1^{k-1}} G_n^{i}.
    \end{equation}
\end{corollary}
\begin{proof}
    Since $s=1-r$, the root $t_2$ of the characteristic polynomial is 
    \[ t_2 = \frac{r - \sqrt{r^2 + 4s}}{2} = \frac{r - \sqrt{r^2 + 4(1-r)}}{2} = \frac{r - (r-2)}{2} = 1. \]
    Thus, the coefficients $a_i(n)$ are
    \begin{align*}
       a_0(n) &= c_2 + \frac{(-c_2)^k }{c_1^{k-1}},   \\
       a_i(n) &= \binom{k}{i} \frac{(-c_2)^{k-i} }{c_1^{k-1}} \quad \text{for} \quad i \in \{1, 2, \ldots, k\}.
    \end{align*}
    These coefficients do not depend on $t_2^n$, and \cref{proofEq1} simplifies to \cref{polyId2}.
\end{proof}

\begin{corollary}
    Let $s = r+1$, so that the gibonacci sequence $(G_n)$ is defined $G_0 = a$, $G_1 = b$, and $G_n = rG_{n-1} + (r+1)G_{n-2}$ for $n \ge 2$.
    Then 
    \begin{equation}\label{polyId3}
        G_{kn} = c_2 (-1)^{kn} + \sum_{i=0}^{k}  (-1)^{(k-i)(n+1)}\binom{k}{i}\frac{c_2^{k-i} }{c_1^{k-1}} G_n^{i}.
    \end{equation}

\end{corollary}
\begin{proof}
    Since $s = r + 1$, the root $t_2$ of the characteristic polynomial is 
    \[ t_2 = \frac{r - \sqrt{r^2 + 4s}}{2} = \frac{r - \sqrt{r^2 + 4(r+1)}}{2} = \frac{r - (r+2)}{2} = -1. \]
    Thus, the coefficients $a_i(n)$ are
    \begin{align*}
       a_0(n) &= c_2 (-1)^{kn} + (-1)^{k(n+1)} \frac{c_2^k }{c_1^{k-1}},   \\
       a_i(n) &= (-1)^{(k-i)(n+1)}\binom{k}{i}\frac{c_2^{k-i} }{c_1^{k-1}} \quad \text{for} \quad i \in \{1, 2, \ldots, k\}.
    \end{align*}
   These coefficients  depend on $(-1)^n$, and \cref{proofEq1} simplifies to \cref{polyId3}.
\end{proof}

\section{Repeated Roots}

\Cref{gibPolyThm,genpoly} apply only to gibonacci sequences whose characteristic polynomials have distinct roots; we have not yet addressed sequences with a repeated root.
It turns out that sequences with a repeated root may satisfy polynomial identities, but only in rather trivial cases.

If a gibonacci sequences $(G_n)$, as defined in \cref{gibSeqDef}, has a characteristic polynomial with a repeated root, then its recurrence must be of the form 
\[ G_n = 2\ell G_{n-1} - \ell^2 G_{n-2} \]
for some integer $\ell$. 
In this case, Binet's formula becomes $G_n = (c_1 + c_2 n)\ell^n$, with $c_1 = a$ and $c_2 = \frac{b-a\ell}{\ell}$.
We find that such a sequence satisfies a polynomial identity if it is in fact a geometric, arithmetic, or alternating arithmetic sequence, as stated in the following proposition.

\begin{proposition}
    Define sequence $(G_n)$ by $G_0 = a$, $G_1 = b$, and $G_n = 2\ell G_{n-1} - \ell^2 G_{n-2}$ for $n \ge 2$, where integers $a$ and $b$ are not both zero, and $\ell$ is a nonzero integer.
    Then $(G_n)$ satisfies polynomial identities in the following cases.
    \begin{enumerate}[label=\bf{Case \arabic*.},leftmargin=46pt]
        \item If $b = a\ell$, then $(G_n)$ is a geometric sequence and
            \begin{equation}\label{repeatedRootsGeoId}
                G_{kn} = \frac{1}{a^{k-1}} G_n^k.
            \end{equation}
        \item If $\ell = 1$, then $(G_n)$ is an arithmetic sequence and
            \begin{equation}\label{repeatedRootsArithId}
                G_{kn} = kG_n + (1-k)a.
            \end{equation}
        \item If $\ell = -1$, then $(G_n)$ is an alternating arithmetic sequence and
            \begin{equation}\label{repeatedRootsAltArithId}
                G_{kn} = (-1)^{(k-1)n}kG_n + (-1)^{kn}(1-k)a.
            \end{equation}
    \end{enumerate}
\end{proposition}

\begin{proof}
    We proceed case by case.

    \textbf{Case 1.} If $b = a\ell$, then $G_n = a\ell^n$ for all nonnegative integers $n$. 
        Thus, $(G_n)$ is a geometric series and satisfies the identity $G_{kn} = \frac{1}{a^{k-1}} G_n^k$.

    \textbf{Case 2.} If $\ell = 1$, then Binet's formula becomes $G_n = a + (b-a)n$.
        This means that $(G_n)$ is an arithmetic sequence. By Binet's formula and simple algebra,
        \begin{align*}
            G_{kn} &= a + (b-a)kn \\
            &= k(a + (b-a)n) - (k-1)a \\
            &= kG_n - (k-1)a,
        \end{align*}
        which is \cref{repeatedRootsArithId}.

    \textbf{Case 3.} If $\ell = -1$, then Binet's formula becomes $G_n = (a -n(a+b))(-1)^n$.
        Thus, $(G_n)$ is an alternating arithmetic sequence.
        \begin{align*}
            G_{kn} &= (a -kn(a+b))(-1)^{kn} \\
            &= (ak -kn(a+b))(-1)^{kn} - ak(-1)^{kn} + a(-1)^{kn} \\
            &= (-1)^{(k-1)n}kG_n + (-1)^{kn}(1-k)a,
        \end{align*}
        which is \cref{repeatedRootsAltArithId}.
        \qedhere
\end{proof}

While the identities in \cref{repeatedRootsGeoId,repeatedRootsArithId,repeatedRootsAltArithId} are in the form of \cref{gibPoly}, these identities are somewhat trivial since they have only one or two nonzero terms. 
Furthermore, \cref{repeatedRootsGeoId} is equivalent to \cref{idCase1}.

\section{Final Thoughts}

We have found polynomial identities for certain families of gibonacci sequences, but questions remain. 
Are there other polynomial identities beyond those stated above? 
In what ways could we generalize these polynomial identities?
Or, in what ways could we generalize our notion of gibonacci sequence that would lead to more polynomial identities?
We leave it to the interested reader to pursue these questions.

\section*{Acknowledgments}

This work was started when the first three authors were undergraduate students at St.\ Olaf College.
We thank the St.\ Olaf Collaborative Undergraduate Research and Inquiry (CURI) for supporting undergraduate research.
In the preparation of this document, generative AI was used only for final grammatical checks and searching for typos.
The mathematical content of this article is entirely due to the authors.

\bibliographystyle{alpha}
\bibliography{references} 

@article{filipponi1995,
    author = {Piero Filipponi},
    title = {Some Binomial {Fibonacci} Identities},
    journal = {The Fibonacci Quarterly},
    volume = {33},
    number = {3},
    pages = {251--257},
    year = {1995},
    publisher = {Taylor \& Francis},
    doi = {10.1080/00150517.1995.12429143}
}

@article{gould1997,
    author = {Henry W. Gould},
    title = {The {Girard-Waring} Power Sum Formulas for Symmetric Functions, and {Fibonacci} Sequences},
    journal = {The Fibonacci Quarterly},
    volume = {37},
    number = {2},
    pages = {135--140},
    year = {1999},
    publisher = {Taylor \& Francis},
    doi = {10.1080/00150517.1999.12428871}
}

@article{halton1965,
    author = {John H. Halton},
    title = {On a General {Fibonacci} Identity},
    journal = {The Fibonacci Quarterly},
    volume = {3},
    number = {1},
    pages = {31--43},
    year = {1965},
    publisher = {Taylor \& Francis},
    doi = {10.1080/00150517.1965.12431452}
}

@article{jennings1993,
    author = {Derek Jennings},
    title = {Some Polynomial Identities for the {Fibonacci} and {Lucas} Numbers},
    journal = {The Fibonacci Quarterly},
    volume = {31},
    number = {2},
    pages = {134--137},
    year = {1993},
    publisher = {Taylor \& Francis},
    doi = {10.1080/00150517.1993.12429302}
}

@book{koshy2017,
    title     = {{Fibonacci} and {Lucas} Numbers with Applications},
    author    = {Thomas Koshy},
    volume    = {1},
    year      = {2017},
    publisher = {Wiley},
    address   = {New York}
}

@article{martinjak2017,
    author = {Ivica Martinjak and Iva Vrsaljko},
    title = {On Polynomial Identities for Recursive Sequences},
    journal = {Miskolc Mathematical Notes},
    volume = {18},
    number = {1},
    pages = {327--336},
    year = {2017},
    publisher = {Miskolc University Press},
    doi = {10.18514/MMN.2017.1776}
}

@article{swamy1997,
    author = {M. N. S. Swamy},
    title = {On Certain Identities Involving {Fibonacci} and {Lucas} Numbers},
    journal = {The Fibonacci Quarterly},
    volume = {35},
    number = {3},
    pages = {230--231},
    year = {1997},
    publisher = {Taylor \& Francis},
    doi = {10.1080/00150517.1997.12428991}
}

\end{document}